\documentclass[12pt,a4paper]{article}
\usepackage{amssymb}
\usepackage{amsmath}
\usepackage{amsthm}
\usepackage{indentfirst}
\usepackage{ot1patch}
\usepackage{psfrag}
\usepackage{rotate}
\usepackage{mathtools}
\usepackage{hyperref}

\newtheorem{tw}{Theorem}[section]
\newtheorem{pr}[tw]{Proposition}
\newtheorem{lm}[tw]{Lemma}
\newtheorem{cor}[tw]{Corollary}
\newtheorem{rem}[tw]{Remark}

\theoremstyle{definition}

\title{K + M constructions with general overrings and relationships with polynomial composites}

\author{\L ukasz Matysiak\\
	Kazimierz Wielki University\\
	Bydgoszcz, Poland \\
	lukmat@ukw.edu.pl}
\title{K + M constructions with general overrings and relationships with polynomial composites}
\begin{document}
	
\maketitle

\begin{abstract}
		In this paper we consider the construction of $K + M$, where $K$ is the domain, $M$ is the maximal ideal of a some ring of polynomials with coefficients from the field $L$, where $K$ is its subring. In addition to the usual domains, we also consider the Noetherian, Pr\"ufer and GCD-domains. In particular, polynomial composites are a case of $K + M$ construction. In this paper we will find numerous construction conclusions related to polynomial composites.
\end{abstract}

\begin{table}[b]\footnotesize\hrule\vspace{1mm}
	Keywords: domain, field, maximal ideal, polynomial.\\
	2010 Mathematics Subject Classification:
	Primary 13B25, Secondary 13G05.
\end{table}

\section{Introduction}
	
D.D.~Anderson, D.F.~Anderson, M. Zafrullah in \cite{xx1} called object $K+XL[X]$ as a composite for $K\subset L$ fields. There are a lot of works where composites are used as examples to show some properties. But the most important works are presented below.

\medskip

In 1976 \cite{y1} authors considered the structures in the form $D+M$, where $D$ is a domain and $M$ is a maximal ideal of ring $R$, where $D\subset R$. Later (in \cite{mm1}), we proved that in composite in the form $D+XK[X]$, where $D$ is a domain, $K$ is a field with $D\subset K$, that $XK[X]$ is a maximal ideal of $K[X]$. 
Next, Costa, Mott and Zafrullah (\cite{y2}, 1978) considered composites in the form $D+XD_S[X]$, where $D$ is a domain and $D_S$ is a localization of $D$ relative to the multiplicative subset $S$. 
Zafrullah in \cite{y3} continued research on structure $D+XD_S[X]$ but 
he showed that if $D$ is a GCD-domain, then the behaviour of $D^{(S)}=\{a_0+\sum a_iX^i\mid a_0\in D, a_i\in D_S\}=D+XD_S[X]$ depends upon the relationship between $S$ and the prime ideals $P$ od $D$ such that $D_P$ is a valuation domain (Theorem 1, \cite{y3}).
In 1991 there was an article \cite{1} that collected all previous composites and the authors began to create a theory about composites creating results. In this paper, the structures under consideration were officially called as composites. 
After this article, various minor results appeared. But the most important thing is that composites have been used in many theories as examples. 
We have researched many properties of composites in \cite{mm1} and \cite{mm2} and \cite{mm3}.

\medskip

Suppose that $T$ is a domain and $L$ is a field that is a retract of $T$, that is, suppose $T=L+M$, where $M$ is a maximal ideal of $T$. Each subring $K$ of $L$ determines a subring $R=K+M$ of $T$. This construction has been studied extensively in two situations. The first systematic investigation of the properties of $R$ is due to R.Gilmer (\cite{8}, Appendix 2, p.558) and Gilmer and W. Heinzer \cite{9}, who required that $T$ be a valuation domain. A similar investigation has been conducted under the hypothesis that $T=L[X]$, $M=XL[X]$, and $L$ is the quotient field of $K$ \cite{4}. The interest in this case arises because $R$ is the symmetric algebra of the $K$-module $L$. In both cases the properties of $R$ are related to those of $K$; in the case of a valuation domain, the relationship of $K$ to $L$ also plays an important role. In this paper, we investigate the costruction described above, without placing any limitations on $T$.

\medskip

More specifically, we focus attention on four properties: we obtain necessary and sufficient conditions for $R$ to be a coherent domain, a Pr\"ufer domain, a Noetherian domain, and a GCD-domain. What is most satisfying is that the conditions are expressed solely in terms of the properties of the components of the construction. If $L$ is the quotient field of $K$, it is also possible to describe the prime ideal lattice of $R$ and thus to compute the Krull dimension of $R$. If $R$ is a Prufer domain, so are $K$ and $T$. Their ideal class groups are shown to be related by a short exact sequence. This yields conditions for $R$ to be a B\'ezout domain. Unfortunately, if $R$ is a Pr\"ufer domain it has the $n$-generator property whenever $K$ and $T$ do. Thus, this construction casts no light on whether invertible ideals in Pr\"ufer domains can require more than two generators. The paper concludes with a brief consideration of methods for obtaining domains $T$ of the form $L+M$ that satisfy the conditions of the theorems.

\medskip

It is undoubtedly possible to characterize other properties. We have limited ourselves to these four because they have received attention in the special contexts investigated earlier, and because they seem adequate to demonstrate that such problems can often be handled in more generality than had previously seemed feasible. It seems quite likely that at least some of the results of this paper can be extended to a somewhat more general situation. As Gilmer \cite{8} noted, the assumption that $L$ is a retract of $T$ is often not essential. Instead, it can be assumed that $L=T/M$, in which case $R$ is replaced by the pullback of $T$ and $K$. However, we have chosen to follow Gilmer's lead in this regard, and for the sake of clarity and simplicity we limit ourselves to the case of a retract.

\medskip

Our interest in this problem was kindled by the paper of D. Dobbs and I. Papick \cite{5}, which gives necessary and sufficient conditions for $R$ to be coherent when $T$ is a valuation domain.

\section{Results}

We have arranged the following considerations and results in such a way that we present $K + M$ constructions with general overrings. After each result, we put the relationship with polynomial composites in the form of a corollary.

\medskip

The letters $T$, $L$, $M$, $K$, and $R$ will retain throughout the paper the meanings assigned to them in the opening paragraph. It being necessary to exclude the situations wherein the construction degenerates, we also assume without further mention that $M\neq 0$ and $K\neq L$.

\begin{lm}
	\label{l1}
	If there exists a nonzero ideal $A$ of $T$ that is finitely generated as an $R$-module, then $K$ is a field and $[L\colon K]<\infty$.
\end{lm}

\begin{proof}
Clearly, $A$ is finitely generated over $T$, and hence $MA\neq A$. For otherwise, $MT_M\cdot AT_M = AT_M$ and therefore $AT_M=0$, by Nakayama’s lemma. This is impossible, since $0\neq A\subseteq AT_M$. It follows that $A/MA$ is a nonzero $(T/M = L)$-module that is finitely generated as an $(R/M = K)$-module. Since $L$ is a field, $A/MA$ can be written as a direct sum of copies of $L$. Thus, $L$ is a finitely generated $K$-module. But then $K$ is a field, since the field $L$ is integral over $K$ and obviously $[L\colon K]<\infty$. 	
\end{proof}

\begin{cor}
	\label{cor1}
	If there exists a nonzero ideal $A$ of $L[X]$, where $L$ be a field, that is finitely generated as an $K+XL[X]$-module, then $K$ is a field and $[L\colon K]<\infty$.
\end{cor}

For the purpose of casting the next result in generality sufficient to cover all the situations that arise, we use the following terminology from \cite{14}. A domain $S$ with quotient field $L$ is called a finite conductor domain if for each pair $x, y\in L$ we have $xS\cap yS$ is a finitely generated $S$-module. Every coherent domain is a finite conductor domain, as is every GCD-domain (\cite{3}, Theorem 2.2) and (\cite{8}, Theorem B, p. 605).

\begin{pr}
	\label{p2}
If $R$ is a finite conductor domain, then exactly one of the following conditions holds:
\begin{itemize} 
	\item[(a) ] $K$ is a field, $[L\colon K]<\infty$, and $M$ is a finitely generated ideal of $T$.
	\item[(b) ] $L$ is the quotient field of $K$ and $T_M$ is a valuation domain.
\end{itemize}	
\end{pr}

\begin{proof}
	If $L$ is not the quotient field of $K$, then there is an $x$ ($0\neq x\in L$) such	that $xK\cap K = (0)$. Clearly, $xR = xK + xM = xK+M$, since $x$ is a unit in $T$ and $M$ is an ideal of $T$. Now $R$ is a finite conductor domain, and $xR$ and $R$ are principal fractional ideals of $R$. Therefore, $xR\cap R$ is a finitely generated $R$-module. But
	$$ xR\cap R = (xK+M)\cap (K+M) = (xK\cap K) + (M\cap M) = M.$$
	Hence, $M$ is a finitely generated ideal of $T$, and by Lemma \ref{l1}, (a) holds.
	
	\medskip
	
	If $L$ is the quotient field of $K$, let $a$ and $b$ be nonzero elements of $T$. Now $aR\cap bR \supseteq aM\cap bM$, and the latter is a nonzero ideal of $T$. Moreover, since $R$ is a finite conductor domain and $L$ is the quotient field of $K$, it follows from Lemma \ref{l1} that
	$$aR\cap bR\neq aM\cap bM.$$
	Choose $x\in (aR\cap bR)\setminus (aM\cap bM)$. Write $x = (k_1 + m_1)a = (k_2 + m_2)b$ with
	$k_1, k_2\in K$ and $m_1, m_2\in M$. One of the elements $k_1$ and $k_2$ is nonzero, say
	$k_1\neq 0$. Since $k_1 + m_1\notin M$, $k_1 + m_1$ is a unit in $T_M$. Therefore,
	$$a = (k_1 + m_1)^{-1} (k_2 + m_2)b\in bT_M,$$
	and thus $aT_M\subseteq bT_M$. It follows that $T_M$ is a valuation domain.
\end{proof}

\begin{cor}
	\label{cor2}
	If $K+XL[X]$ is a finite conductor domain, then exactly one of the following conditions holds:
		\begin{itemize} 
			\item[(a) ] $K$ is a field, $[L\colon K]<\infty$, and $XL[X]$ is a finitely generated ideal of $T$.
			\item[(b) ] $L$ is the quotient field of $K$ and $L[X]_{XL[X]}$ is a valuation domain.
		\end{itemize}	
\end{cor}

This is a convenient juncture for recording some observations that we shall use frequently.

\medskip

$R$ and $T$ have the same quotient field. This is a general remark about integral domains that have a nonzero ideal in common.

\medskip

If $T$ is integrally closed, the integral closure of $R$ is $J + M$, where $J$ is the integral closure of $K$ in $L$. This follows easily from the fact that $R$ and $T$ have the same quotient field.

\medskip

If $K$ is a field and $[L\colon K]<\infty$ then $T$ is a finite $R$-module. Indeed, if $\{1, b_2, \dots, b_n\}$ is a field basis for $L/K$, then $\{1, b_2, \dots, b_n\}$ is an $R$-module
generating set for $T$.

\medskip

If $K$ is the quotient field of $D$, then $T= R_{K\setminus\{0\}}$ is a localization of $R$. Moreover, $R$ is a faithfully flat $K$-module. That no maximal ideal of $K$ blows up in $R$
is obvious; moreover, as a $K$-module, $R$ is the direct sum of $K$ and a $K$-module, namely $M$, which is a direct sum of copies of $L$, a flat $K$-module.

\medskip

We come now to our first theorem. Recall that a domain $S$ is coherent if direct products of flat $S$-modules are flat. Other characterizations include “finitely generated ideals are finitely presented” and “any two finitely generated ideals of $S$ have finite intersection” (\cite{3}, Theorems 2.1 and 2.2). Thus, Noetherian domains and Prufer domains are coherent.

\begin{tw}
\label{t3}
The following conditions are equivalent:
\begin{itemize} 
	\item[(1) ] $R$ is coherent.
	\item[(2) ] $T$ is coherent and exactly one of the following holds:
\begin{itemize} 
	\item[(a) ] $M$ is a finitely generated ideal of $T$, $K$ is a field, and $[L\colon K]<\infty$.
	\item[(b) ] $L$ is the quotient field of $K$, $K$ is coherent, and $T_M$ is a valuation domain.	
\end{itemize}
\end{itemize} 
\end{tw}

\begin{proof}

($\Rightarrow$)\ \ By Proposition \ref{p2}, two cases arise.

\medskip

If $K$ is a field, $[L\colon K]<\infty$, and $M$ is a finitely generated ideal of $T$, then $T$ is a finite $R$-module. It follows from (\cite{11}, Corollary 1.5, p. 476) that $T$ is coherent.

\medskip

If $L$ is the quotient field of $K$, then $T$, being a localization of $R$, is coherent, and by Proposition \ref{p2}, $T_M$ is a valuation domain. To see that $K$ is coherent, one can show directly that the intersection of two finitely generated ideals is finitely generated, or, given a finitely generated ideal $I$ of $K$, one can use the faithful flatness of $R$ over $K$ to descend the finite presentation of $IR = I\otimes_K R$ to a finite presentation of $I$.

\medskip

\noindent 
($\Leftarrow$)\ \ Suppose the conditions of (a) hold. We shall need the following general remark. Let $\{1, b_2, \dots, b_n\}$ be a field basis for $L/K$. If $\phi$ is the $R$-homomorphism from $R^n$ to $T$ given by $\phi(r_1, \dots, r_n)=\sum r_ib_i$, then $\phi$ is surjective and the kernel of $\phi$ is isomorphic to $M^{n-1}$. The only statement that needs justification is the one about the kernel. Write $r_i = k_i + m_i$. Then $\phi(r_1, \dots, r_n) = 0$ if and only if $\sum (k_i + m_i)b_i = 0$, which implies $\sum k_ib_i = 0$; thus $k_1 = \dots = k_n = 0$, since
$1, b_2, \dots, b_n$, are $K$-linearly independent. Also, $\sum m_ib_i = 0$, which entails $m_1 = -\sum_2^n m_ib_i$. The isomorphism from $M^{n-1}$ onto the kernel of $\phi$ is given by
$$(m_2, \dots, m_n)\mapsto \big( -\sum_2^n m_ib_i, m_2, \dots, m_n\big).$$

Therefore,

$$0\rightarrow M^{n-1}\rightarrow R^n\rightarrow T\rightarrow 0$$

\noindent
is a presentation of $T$ as an $R$-module, and since $M$ is finitely generated, $T$ is a finitely presented $R$-module. To show that $R$ is coherent, we shall argue that direct products of flat $R$-modules are flat. Thus, let $\{E_{\alpha}\}$ be a collection of flat $R$-modules. For each $a$, the $T$-module $E_{\alpha}\otimes_R T$ is $T$-flat, and since $T$ is coherent, $\prod_{\alpha} (E_{\alpha}\otimes_R T)$ is $T$-flat. But since $T$ is finitely presented, 
$$\prod_{\alpha} (E_{\alpha}\otimes_R T)\cong (\prod_{\alpha} E_{\alpha})\otimes_R T$$

(\cite{1}, Exercise 9, p. 43). By the descent lemma of D. Ferrand (\cite{7}, p. 946), $\prod_{\alpha} E_{\alpha}$ is $R$-flat. 

\medskip

Suppose the conditions of (b) hold. Since $T_M$ is a valuation domain, $C/MC$ is a $L$-vector space of dimension at most $1$ for each ideal $C$ of $T$. Indeed, let $a, b\in C\setminus MC$. Then either $a/b$ or $b/a$ lies in $T_M$, say
$$a/b = t/(l+m)\quad (t\in T, m\in M, l\in L\setminus\{0\}).$$

Then $(l+m)a = bt$ or, what is the same thing, $a =l^{-1}tb-l^{-1}ma$. Thus $a = (l^{-1}t)b + MC$.

\medskip

Now, let $A$ and $B$ be nonzero, finitely generated ideals of $R$. Then $AT\cap BT = (A\cap B)T$ is finitely generated, say by $c_1, \dots, c_n\in A\cap B$. This is possible since $T$ is a localization of $R$. Since
$$R(c_1, \dots, c_n)\supseteq MR(c_1, \dots, c_n)=MT(c_1, \dots, c_n)=MT(A\cap B)=M(A\cap B),$$

\noindent 
if we can show that $(A\cap B)/M(A\cap B)$ is finitely generated over $R$, it will follow that $A\cap B$ is finitely generated. It is clearly sufficient to prove that $(A\cap B)/M(A\cap B)$ is finitely generated over $K$. Therefore, consider the exact sequence
$$0\rightarrow A\cap B\rightarrow A\oplus B\rightarrow A+B\rightarrow 0.$$

Tensoring with $R/M = K$, we obtain the exact sequence
	
$$(A\cap B)/M(A\cap B)\xrightarrow[]{\alpha} (A/MA)\oplus (B/MB)\xrightarrow[]{\beta} (A+B)/M(A+B)\rightarrow 0.$$
	
We claim $\alpha$ is monic. Tensoring
	
$$(A/MA)\oplus (B/MB)\xrightarrow[]{\beta}(A+B)/M(A+B)\rightarrow 0$$

\noindent 
with $T$ or $L$, we see that the sequence

$$(TA/MA)\oplus (TB/MB)\xrightarrow[]{T\otimes\beta}T(A+B)/M(A+B)\rightarrow 0$$

\noindent 
is exact. Since $A$ and $B$ are finitely generated, the $L$-dimension of $TC/MC$ is $1$ for $C=A$, $B$, or $A+B$. Therefore $T\otimes\beta$ is not monic, and therefore $\beta$ is not monic. Now the kernel of $\beta$ is a nonzero submodule of the torsion-free $K$-module $(A/MA)\oplus (B/MB)$, each factor being embeddable in $L$, the quotient field of $K$. But $\alpha$ maps onto the kernel of $\beta$, and $(A\cap B)/M(A\cap B)$ is embedded in one copy of $L$. This proves that $\alpha$ is monic.

\medskip

Now the $K$-modules at both ends of $\beta$ are finitely generated submodules of direct sums of copies of $L$, and consequently they are finitely presented (\cite{3}, Theorem 2.1). It follows that the kernel $(A\cap B)/M(A\cap B)$ of $\alpha$ is finitely generated (\cite{1}, Lemma 9, p. 21).
\end{proof}

\begin{cor}
	\label{cor3}
	The following conditions are equivalent:
	\begin{itemize} 
		\item[(1) ] $K+XL[X]$ is coherent.
		\item[(2) ] $L[X]$ is coherent and exactly one of the following holds:
	\begin{itemize} 
		\item[(a) ] $XL[X]$ is a finitely generated ideal of $L[X]$, $K$ is a field, and $[L\colon K]<\infty$.
		\item[(b) ] $L$ is the quotient field of $K$, $K$ is coherent, and $L[X]_{XL[X]}$ is a valuation domain.	
	\end{itemize}
\end{itemize} 
\end{cor}

The Noetherian case is much easier to handle.

\begin{tw}
	\label{t4}
	The following conditions are equivalent:
	\begin{itemize} 
		\item[(1) ] $R$ is Noetherian.
		\item[(2) ] $T$ is Noetherian and $K$ is a field with $[L\colon K]<\infty$.
	\end{itemize} 	
\end{tw}

\begin{proof} 
	($\Rightarrow$) \ \ Since $M$ is a finitely generated ideal of $R$, it follows from Lemma \ref{l1} that $K$ is a field and $[L\colon K]<\infty$. Thus, $T$ is module-finite over the Noetherian ring $R$.
	
	\medskip
	
	\noindent 
	($\Leftarrow$)\ \  $T$ is a Noetherian ring and module-finite over the subring $R$. This is the situation covered by P. M. Eakin’s Theorem \cite{6}.
\end{proof}

\begin{cor}
	\label{cor4}
	The following conditions are equivalent:
	\begin{itemize} 
		\item[(1) ] $K+XL[X]$ is Noetherian.
		\item[(2) ] $[L\colon K]<\infty$.	
	\end{itemize} 
\end{cor}

The Pr\"ufer-domain case also presents little difficulty.

\begin{tw}
	\label{t5}
	The following conditions are equivalent:
	\begin{itemize} 
		\item[(1) ] $R$ is a Pr\"ufer domain.
		\item[(2) ] $T$ is a Pr\"ufer domain, $L$ is the quotient field of $K$, and $K$ is a Pr\"ufer domain.
	\end{itemize} 	
\end{tw}

\begin{proof}
	($\Rightarrow$)\ \  Since $T$ is a localization of $R$, $T$ is a Pr\"ufer domain. Moreover, since Pr\"ufer domains are integrally closed, $L$ is the quotient field of $K$, by Proposition \ref{p2}. That finitely generated ideals of $K$ are invertible may be seen directly, or one can argue this, using the fact that the faithfully flat $K$-module $R$ is a Pr\"ufer domain.
	
	\medskip
	
	\noindent 
	($\Leftarrow$)\ \ Given a finitely generated nonzero ideal $I$ of $R$, we must show that $I$ is a
	projective $R$-module (\cite{2}, Proposition 3.2, p. 132). What comes to the same thing, since $R$ is a domain (\cite{12}, Corollary 3.2, p. 108), is to show that $I$ is a flat $R$-module. Now $IT = I\otimes_R T$ is $T$-projective, since $T$ is a Pr\"ufer domain. Moreover,
	$$0\neq I/MI\subseteq IT/MTI\cong IT\otimes_T (T/M)\cong IT\otimes_T L,$$
	a $L$-vector space. In particular, $I/MI$ is a torsion-free $K$-module, and $K$ is a Pr\"ufer domain. Consequently, $I/MI$ is $K$-flat, and it follows from the descent lemma of Ferrand that $I$ is $R$-flat (\cite{7}, p. 946).
\end{proof}

\begin{cor}
	\label{cor5}
	The following conditions are equivalent:
	\begin{itemize} 
		\item[(1) ] $K+XL[X]$ is a Pr\"ufer domain.
		\item[(2) ] $L[X]$ is a Pr\"ufer domain, $L$ is the quotient field of $K$, and $K$ is a Pr\"ufer domain.
	\end{itemize} 
\end{cor}

Recall that the class group $\mathcal{C}(S)$ of a Pr\"ufer domain $S$ is the multiplicative group of invertible fractional ideals of $S$ modulo the subgroup of nonzero principal fractional ideals. The class group may also be regarded as the multiplicative group of isomorphism classes of invertible fractional ideals of $S$. In the construction of this paper, the class groups of the Prufer domains $R$, $K$, and $T$ are nicely related.

\begin{pr}
	\label{p6}
	If $R$ is a Pr\"ufer domain, there exists an exact sequence
	$$1\rightarrow\mathcal{C}(K)\xrightarrow[]{\alpha}\mathcal{C}(R)\xrightarrow[]{\beta}\mathcal{C}(T)\rightarrow 1,$$
	where $\alpha[J]=[JR]$ and $\beta[I]=[IT]$ for all finitely generated fractional ideals $J$ of $K$ and $I$ of $R$. Here, $[I]$ denotes the isomorphism class of the ideal $I$.
\end{pr}

\begin{proof}
	Clearly, $\alpha$ and $\beta$ are well-defined homomorphisms. Also, since the quotient field $L$ of $K$ is contained in $T$, $\beta\alpha[J] = [JRT] = [JT] = [JLT] = [T]$ for each
	fractional ideal $J$ of $K$. Thus $\alpha(\mathcal{C}(K))\subseteq \ker\beta$. Suppose $\alpha[I]\in\ker\beta$. We may assume $I$ is an integral ideal of $R$, since $[I]$ has such a representative. Therefore $IT = xT$, with $x\in T$. Since $T=R_{K\setminus\{0\}}$ we may choose $x\in I$, say $x = k + m$ with $k\in K$ and $m\in M$. Suppose $I = R(k_1 +m_1) + \dots + R(k_t + m_t)$ with $k_i\in K$ and $m_i\in M$. Then
	$$k_i + m_i = (l_i + n_i)(k + m) = l_i(k+m) + n_i(k + m),$$
	and therefore
	$$l_i(k+m) = (k_i + m_i) - n_i(k+m)\in I.$$
	Therefore, $I\supseteq R(K_{l_1} + \dots + K_{l_t})(k+m)$. Since $M$ is an ideal of $T$, we see that $l_j^{-1}M\subseteq M$ and hence $R_{l_l}\supseteq l_jM\supseteq M$. It follows that
	$$R(K_{l_1} + \dots + K_{l_t})(k+m)\supseteq M(k+m).$$
	Hence, for $1\leqslant i\leqslant t$,
	$$k_i+m_i = l_i(k + m)+n_i(k+m)\in R(K_{l_1}+\dots +K_{l_t})(k + m).$$
	But these elements generate $I$. Therefore, $I = R(K_{l_1}+\dots +K_{l_t})(k + m)$. Hence
	$$[I] = [R(K_{l_1}+\dots +K_{l_t})] = \alpha[K_{l_1} + \dots + K_{l_t}].$$
	It follows that $\ker\beta\subseteq\alpha(\mathcal{C}(K))$ and therefore that $\ker\beta\subseteq\alpha(\mathcal{C}(K))$.
	
	\medskip
	
	It follows immediately from the relation $T = R_{K\setminus\{0\}}$ that $\beta$ is an epimorphism.
	It remains only to show that $\alpha$ is monic. Suppose that $\alpha [J]=[R]$. As before, we
	may assume that $J$ is an integral ideal of $K$. Since $JM = M$ as above,
	$JR = JK + JM = J + M$. Thus, $J + M = (K + M)(k + m) = Kk + M$, because
	$(K + M)(k + m)\supseteq M$. It follows that $J = Kk$, since both sums are direct. Therefore,
	$[J] = [K]$ and $\alpha$ is monic.
\end{proof}

\begin{cor}
	\label{cor6}
	If $K+XL[X]$ is a Pr\"ufer domain, there exists an exact sequence
	$$1\rightarrow\mathcal{C}(K)\xrightarrow[]{\alpha}\mathcal{C}(K+XL[X])\xrightarrow[]{\beta}\mathcal{C}(L[X])\rightarrow 1,$$
	where $\alpha[J]=[J(K+XL[X])]$ and $\beta[I]=[I(L[X])]$ for all finitely generated fractional ideals $J$ of $K$ and $I$ of $K+XL[X]$. Here, $[I]$ denotes the isomorphism class of the ideal $I$.
\end{cor}

\begin{rem}
Our proof shows that when $T$ is a B\'ezout domain, each finitely generated ideal $I$ of $R$ can be written in the form $I = (l_1, \dots, l_n, 1)\cdot b\cdot R$ ($l_i\in L, b\in I$). In particular, $I$ is $R$-isomorphic to the extension of a finitely generated ideal of $K$. This is often useful in cases where $T$ is a valuation domain or $L[X]$.	
\end{rem}

\begin{tw}
	\label{t7}
The following conditions are equivalent:
\begin{itemize} 
	\item[(1) ] $R$ is a B\'ezout domain.
	\item[(2) ] $T$ is a B\'ezout domain, $L$ is the quotient field of $K$, and $K$ is a B\'ezout domain.	
\end{itemize} 
\end{tw}

\begin{proof}
Because B\'ezout domains are precisely the Pr\"ufer domains having trivial class group, the result follows from Theorem \ref{t5} and Proposition \ref{p6}.	
\end{proof}

\begin{cor}
	\label{cor7}
	The following conditions are equivalent:
	\begin{itemize} 
		\item[(1) ] $K+XL[X]$ is a B\'ezout domain.
		\item[(2) ] $L$ is the quotient field of $K$, and $K$ is a B\'ezout domain.	
	\end{itemize} 
\end{cor}

We can easily describe the prime-ideal lattice of $R$ in case $L$ is the quotient field of $K$.

\begin{pr}
	\label{p8}
Let $L$ be the quotient field of $K$. If $Q$ is a prime ideal of $R$, then either $Q=P\cap R$ for some prime ideal $P$ of $T$, or $Q = P + M$ for some prime ideal $P$ of $K$.	
\end{pr}

\begin{proof}
Because $T = R_{K\setminus\{0\}}$, there exists a one-to-one correspondence between primes of $R$ that miss $K\setminus\{0\}$ and primes of $T$. On the other hand, if $Q\cap K\neq (0)$,
let $k\in Q\cap K$, $k\neq 0$. For $m\in M$, $k^{-1}m\in M\subseteq R$, and hence $m\in Q$. Therefore, $Q\subseteq M$. But $R/M=K$.	
\end{proof}

\begin{cor}
	\label{cor8}
	Let $L$ be the quotient field of $K$. If $Q$ is a prime ideal of $K+XL[X]$, then either $Q=P\cap K+XL[X]$ for some prime ideal $P$ of $L[X]$, or $Q = P + XL[X]$ for some prime ideal $P$ of $K$.	
\end{cor}

Consequently, the lattice of prime ideals of $R$ looks like the lattice of prime ideals of $K$ “pasted” at $M$ to that of $T$. This gives the following result.

\begin{cor}
	\label{corcor}
	If $L$ is the quotient field of $K$, and if $K$ and $T$ have finite Krull dimension, then $R$ has finite Krull dimension equal to 	
	$$\max \{height_T(M) + \dim (K), dim (T)\}.$$
\end{cor}

\begin{cor}
	\label{cor9}
	If $L$ is the quotient field of $K$, and if $K$ and $L[X]$ have finite Krull dimension, then $K+XL[X]$ has finite Krull dimension equal to 	
	$$\max \{height_{L[X]} XL[X]) + \dim (K), dim (L[X])\}.$$
\end{cor}

A Pr\"ufer domain $S$ is said to have the $n$-generator property if each finitely generated ideal of $S$ can be generated by $n$ or fewer elements. It is an open question whether all Pr\"ufer domains have the $2$-generator property. As the following result shows, the construction of this paper fails to shed new light on this question.

\begin{tw}
	\label{t10}
	The following conditions are equivalent:
	\begin{itemize} 
		\item[(1) ] $R$ is an $n$-generator Pr\"ufer domain.
		\item[(2) ] $T$ and $K$ are $n$-generator Pr\"ufer domains.	
	\end{itemize} 
\end{tw}

\begin{proof}
($\Rightarrow$)\ \  By Theorem \ref{t5}, $T$ and $D$ are Pr\"ufer domains. Moreover, $T$ is a 	localization of $R$ and $D$ is a homomorphic image of $R$.

\medskip

\noindent 
($\Leftarrow$)\ \ By Theorem \ref{t5}, $R$ is a Prufer domain. Let $I = (a_1, \dots, a_t)$ be a nonzero finitely generated ideal of $R$. Since $R_M$ is a valuation domain, $IR_M$ is principal
generated by some $a_j$, which we may assume to be $a_s$. Then there exists $u\notin M$ such that for $2\leqslant j\leqslant t$, $a_j = (r_j/u)a_s$ with $r_s\in R$. Thus, $I$ is $R$-isomorphic to $I(u/a_s) = Ru+ Rr_2 + \dots + Rr_t$, which is an ideal of $R$ not contained in $M$. It is therefore harmless to assume that $I\not\subseteq M$.

\medskip

Thus, $(I + M)/M$ is a nonzero $(R/M)$-submodule of $R/M = K$. Since $K$ is an $n$-generator Pr\"ufer domain,
$$I+M = R(k_1 + m_1) + \dots +R(k_n+ m_n) + M$$
with $k_i\in K$, $k_i\neq 0$, and $m_i\in M$ for $1<i<n$. Because $T$ is also an $n$-generator
Pr\"ufer domain, $IT = T(l_1 + m_1) + \dots + T(l_n + m_n')$ with $l_i\in L$ and $m_i'\in M$. Now, since $I\not\subseteq M$, some $l_i\neq 0$, say $l_s\neq 0$. We may assume, in fact, that $l_i\neq 0$ for each $i$; for if this is not already the case, we can replace $l_i+m_i'$ with $(l_s + m_s) + (l_i + m_i')$. But then
$$IT = T(k_1+m_1'')+\dots + T(k_n+m_n''),$$
where $m_i''=l_i^{-1}k_im_i''$ since $k_i+m_i''=k_il_i^{-1}(l_i+m_i')$ and $k_il_i^{-1}$ is a unit in $T$.

\medskip

We claim that $I = (k_1+m_1'', \dots, k_n+m_n'')R$. It suffices to verify that this equality holds locally at each maximal ideal $P$ of $R$. By Proposition \ref{p8}, there are two types of maximal ideals, those that contain $M$ and those that have trivial intersection with $K$. If $P$ is a maximal ideal of $R$ with $P\cap K = (0)$, then $R_{PT} = T_{PT}$, and the desired equality certainly holds since it already holds when the ideals are extended to $T$. Now suppose that $P$ is a maximal ideal of $R$ with $P\supset M$. Before considering what happens in this case, note that if $A$ is an ideal of $R$ not contained in $M$, then $AR_P\supset MR_P$. This is because these ideals must be comparable, since $R_P$ is a valuation domain. But if $MR_P\supseteq A$, then
$$A\subseteq AR_P\cap R\subseteq MR_P\cap R=M,$$
since $M$ is prime in $R$. In particular, if $k\in K$, $k\neq 0$, and $m\in M$, then $(k+m)R_P\supseteq MR_P$ and $kR_P\supseteq MR_P$. Therefore, $(k+ m)R_P=kR_P$. It is apparent
from these observations that
\begin{align*} 
IR_P &= IR_P + MR_P = (k_1+m_1)R_P + \dots + (k_n + m_n)R_P + MR_P \\
&= (k_1+m_1)R_P + \dots + (k_n + m_n)R_P = (k_1, \dots, k_n)R_P\\ 
&= (k_1 + m_1'', \dots, k_n + m_n'')R_P.
\end{align*} 
\end{proof}

\begin{cor}
	\label{cor10}
	The following conditions are equivalent:
	\begin{itemize} 
		\item[(1) ] $K+XL[X]$ is an $n$-generator Pr\"ufer domain.
		\item[(2) ] $L[X]$ and $K$ are $n$-generator Pr\"ufer domains.	
	\end{itemize} 
\end{cor}

Our next result concerns GCD-domains, and we refer the reader to (\cite{8}, Appendix 4, p. 601) and (\cite{13}, pp. 32-33) for the relevant facts. We remark that Bezout domains and unique-factorization domains afford the most common examples of GCD-domains.

\medskip

We shall adopt the following notation. Let $S$ be a domain, and suppose that $B$ is a torsion-free $S$-module. If $0\neq b_1, b_2\in B$, we shall write $c = \inf_S(b_1,b_2)$ provided $c\in B$, $Sc\supseteq Sb_1+Sb_2$, and $Sc$ is the minimal principal $S$-submodule of $B$ containing $Sb_1 + Sb_2$. When $\inf_S(b_1, b_2)$ exists, it is unique to within multiplication by a unit of $S$. It is easily verified that $S$ is a GCD-domain if and only if
$\inf_S(l_1, l_2)$ exists for all $0\neq l_1, l_2$ belonging to the quotient field of $S$.

\begin{tw}
	\label{t11}
	The following conditions are equivalent:
	\begin{itemize} 
		\item[(1) ] $R$ is a GCD-domain.
		\item[(2) ] $T$ is a GCD-domain, $L$ is the quotient field of $K$, $K$ is a GCD-domain, and $T_M$ is a valuation domain.
	\end{itemize} 	
\end{tw}

\begin{proof}
	($\Rightarrow$)\ \ By Proposition \ref{p2}, since GCD-domains are integrally closed (\cite{13}, 	Theorem 50, p. 33), $L$ is the quotient field of $K$ and $T_M$ is a valuation domain.
	Moreover, $T$, being a localization of $R$, is a GCD-domain, and $K$, being a retract of
	$R$, is also a GCD-domain.
	
	\medskip
	
	\noindent 
	($\Leftarrow$)\ \ We begin with two observations.
	
	\medskip
	
	First, if $S$ is an integral domain contained in a field $M$ and if $0\neq a, b, c\in M$
	are such that $\inf_S(b, c)$ exists, then $\inf_S(ab, ac)$ exists and is equal to $a\cdot \inf_S(b, c)$.
	
	\medskip
	
	To see this, note that the map $a\colon M\to M$ is an $S$-homomorphism. Moreover, the map
	is monic, because $M$ is $S$-torsion-free and epic since $M$ is $S$-divisible. Thus, the
	map is an $S$-automorphism of $M$. Consequently, it induces a one-to-one inclusion-preserving correspondence between the $S$-submodules of $M$ containing $Sb$ and $Sab$ and also between $Sc$ and $Sac$. Since corresponding submodules are isomorphic, they
	require the same number of $S$-generators. Therefore, if there is a principal $S$-submodule of $M$ minimal over $Sb +Sc$, then there is a principal $S$-submodule of $M$ minimal over $Sab + Sac$ and it is induced by multiplication by $a$.
	
	\medskip
	
	Second, let $t_1=l_1+m_1$ and $t_2=l_2+m_2$ belong to $T$, with $m\in M$ and $l_i\in L$
	and not both $l_1$, $l_2$ equal to zero. Because $T$ is a GCD-domain, $\inf_T(t_1, t_2)$ exists and has the form $l_3 + m_3$, with $m_3\in M$ and $l_3\in L$, $l_3\neq 0$, since $Tt_1 + Tt_2\not\subseteq M$.
	Furthermore, $\inf_D(l_1, l_2) = l\neq 0$ exists because $K$ is a GCD-domain with quotient
	field $L$. Because $ll_3^{-1}$ is a unit in $T$, we may assume that 
	$$\inf_T(t_1, t_2) = l+m$$
	with $m\in M$. For this choice, a straight-forward calculation shows that $l+m = inf_R (t_1, t_2)$.
	
	\medskip
	
	Now, let $a$ and $b$ be nonzero elements of $R$. Since $T_M$ is a valuation domain, $aT_M$, and $bT_M$ are comparable, say $a = (t/u)\cdot b$ with $t\in T$ and $u\in T\setminus M$, both nonzero. Thus, $ua = tb$. By our first observation, if $\inf_R(ta, tb)$ exists, then
	$\inf_R(a, b)$ exists and is equal to $(1/t)\cdot\inf_R(ta, tb)$. But $\inf_R(ta, tb) = \inf_R(ta, ua)$. Again applying our first observation, we see that if $\inf_R(t, u)$ exists, then $\inf_R(at, au)$ exists and is equal to $a\cdot \inf_R(t, u)$. By our second observation, $\inf_R(t, u)$ does exist since $u\notin M$, so that $u = m + l$, where $l\neq 0$. This completes the proof.
\end{proof}

\begin{cor}
	\label{cor11}
	The following conditions are equivalent:
	\begin{itemize} 
		\item[(1) ] $K+XL[X]$ is a GCD-domain.
		\item[(2) ] $L$ is the quotient field of $K$, $K$ is a GCD-domain, and $L[X]_{XL[X]}$ is a valuation domain.	
	\end{itemize} 
\end{cor}

We conclude by describing two large classes of domains, different from those previously studied, admitting an arbitrary field $L$ as retract, and to which the program of this paper can be applied.

\medskip

Let $S$ be an abelian, torsion-free, cancellative semigroup with $0$, and $L$ a field.
The semigroup rings $L[S]$ can be regarded as generalizations of polynomial rings.
Conditions on $S$ for $L[S]$ to be a Pr\"ufer domain, B\'ezout domain, GCD domain, and so forth are given in \cite{10}. Moreover, $L[S]$ contains a maximal ideal $M$, the so-called augmentation ideal, with the property that $L[S]=L+M$. For most choices of $S$, $L[S]$ is neither Noetherian nor a polynomial ring, and $L[S]$ is never a valuation domain.

\medskip

Finally, let $K$ be an algebraically closed field, and let $L$ be a field of algebraic functions of a single variable over $K$ having positive genus. It is well known that by intersecting all but one of the DVR’s on $L$ that contain $K$, one obtains a Dedekind domain $S$ having infinite class group and the additional property that $S=K+M$ for each maximal ideal $M$ of $S$.

\section{Declarations}

\noindent
This paper was not funded.

\medskip 

\noindent
Author declares that there is no conflict of interest.

\medskip

\noindent
Data available within the article or its supplementary materials.

\end{document}